\documentclass[11pt]{amsart}
\usepackage{amsfonts, amssymb}
\usepackage{amsmath}
\usepackage{amsthm}
\usepackage{cite}
\numberwithin{equation}{section}

\newtheorem{theorem}{Theorem}[section]
\newtheorem{lemma}[theorem]{Lemma}

\begin{document}

\title[Solutions of differential equations]{On entire function $e^{p(z)}\int_0^{z}\beta(t)e^{-p(t)}dt$ with applications to Tumura--Clunie equations and complex dynamics}

\author{Yueyang Zhang}
\address{School of Mathematics and Physics, University of Science and Technology Beijing, No.~30 Xueyuan Road, Haidian, Beijing, 100083, P.R. China}
\email{zyynszbd@163.com}
\thanks{The author is supported by a Project funded by China Postdoctoral Science Foundation~{(2020M680334)} and the Fundamental Research Funds for the Central Universities~{(FRF-TP-19-055A1)}.}

%\author{Zongsheng Gao}
%\address{LMIB $\&$ School of Mathematics and Systems Science, Beihang University, Beijing 100191, People¡¯s Republic of China}
%\email{zshgao@buaa.edu.cn}

%\author{Jilong Zhang}
%\address{LMIB $\&$ School of Mathematics and Systems Science, Beihang University, Beijing 100191, People¡¯s Republic of China}
%\email{jlzhang@buaa.edu.cn}



\date{\today}

\commby{}

\begin{abstract}
Let $p(z)$ be a nonconstant polynomial and $\beta(z)$ be a small entire function of $e^{p(z)}$ in the sense of Nevanlinna. We describe the growth behavior of the entire function $H(z):=e^{p(z)}\int_0^{z}\beta(t)e^{-p(t)}dt$ in the complex plane $\mathbb{C}$. As an application, we solve entire solutions of Tumura--Clunie type differential equation $f(z)^n+P(z,f)=b_1(z)e^{p_1(z)}+b_2(z)e^{p_2(z)}$, where $b_1(z)$ and $b_2(z)$ are nonzero polynomials, $p_1(z)$ and $p_2(z)$ are two polynomials of the same degree~$k\geq 1$ and $P(z,f)$ is a differential polynomial in $f$ of degree $\leq n-1$ with meromorphic functions of order~$<k$ as coefficients. These results allow us to determine all solutions with relatively few zeros of the second-order differential equation $f''-[b_1(z)e^{p_1(z)}+b_2(z)e^{p_2(z)}+b_3(z)]f=0$, where $b_3(z)$ is a polynomial. We also prove a theorem on certain first-order linear differential equation related to complex dynamics.

\end{abstract}

\maketitle

%\newpage

\section{Introduction}\label{introduction}
In this paper, a meromorphic function always means a function meromorphic in the complex plane~$\mathbb{C}$. We assume that the readers are familiar with the standard notation and basic results of Nevanlinna theory; see, e.g., \cite{Hayman1964Meromorphic,Laine1993}. Moreover, we say that a function $\gamma(z)\not\equiv\infty$ is a \emph{small} function of $f$ if $T(r,\gamma)=S(r,f)$, where $S(r,f)$ denotes any quantity satisfying $S(r,f)=o(T(r,f))$, $r\rightarrow{\infty}$, possibly outside an exceptional set of finite linear measure. For a \emph{differential polynomial} $P(z,f)$ in $f$, we mean a finite sum of monomials in $f$ and its derivatives with the form
\begin{equation*}
P(z,f)=\sum_{j=1}^{n}a_{j}f^{n_{j0}}(f')^{n_{j1}}\cdots(f^{(k)})^{n_{jk}},
\end{equation*}
where $n_{j0},\cdots,n_{jk}\in \mathbb{N}$ and the coefficients $a_j$ are small functions of $f$. We define the \emph{degree} of $P(z,f)$ to be the greatest integer of $d_j:=\sum_{l=0}^kn_{jl}$, $1\leq j\leq n$, and denote it by $\deg_f(P(z,f))$.

A generalization of the theorem of Tumura--Clunie \cite{Clunie1962meromorphic,Tumura1937} given by Hayman~\cite[p.~69]{Hayman1964Meromorphic} states that: If a nonconstant meromorphic function $f(z)$ satisfies $f(z)^n+P(z,f)=g(z)$, where $n\geq 2$ is an integer and $P(z,f)$ is a differential polynomial in $f$ of degree $\leq n-1$ with small functions of $f$ as coefficients, and $N(r,f)+N(r,1/g)=S(r,f)$, then there is a small function $\gamma(z)$ of $f(z)$ such that $(f(z)-\gamma(z))^n=g(z)$. In~\cite{Liping2008,Liping2011,LipingYang2006}, Li and his coauthor considered the equation
\begin{equation}\label{EQ1intro}
f^n+P(z,f)=b_1e^{\alpha_1z}+b_2e^{\alpha_2z},
\end{equation}
where $P(z,f)$ is a differential polynomial in $f$ of degree $\leq n-1$ with small functions of $f$, $b_1$, $b_2$, $\alpha_1$ and $\alpha_2$ are nonzero constants. They solved meromorphic solutions such that $N(r,f)=S(r,f)$ of \eqref{EQ1intro} under assumptions that $\deg_f(P(z,f))\leq n-2$ or that $\deg_f(P(z,f))=n-1$ and $\alpha_1+\alpha_2=0$ or $0<\alpha_2/\alpha_1\leq(n-1)/n$. Li~\cite{Liping2011} further asked how to solve the general solutions of~\eqref{EQ1intro}. Denote by $\rho(f)$ the order of a meromorphic function $f(z)$. Assuming that in \eqref{EQ1intro} all coefficients of $P(z,f)$ have order less than $\rho(f)$, Zhang, Gao and Zhang~\cite{zhang2021} showed that if $\alpha_2/\alpha_1$ is real and \eqref{EQ1intro} has an entire solution $f$, then $\alpha_2/\alpha_1$ must be equal to $-1$ or a positive rational number and in either case $f$ is a linear combination of exponential functions with certain coefficients plus some entire function of order less than~$1$. We note that an application of the method in \cite{Ishizalikazuya19971} to equation \eqref{EQ1intro} will yield that \eqref{EQ1intro} cannot have entire solutions when $\alpha_2/\alpha_1$ is not real and thus~\cite[Theorem~1.1]{zhang2021} actually gives a complete answer to Li's question in the entire solution case.

The main purpose of this paper is to provide a generalization of~\cite[Theorem~1.1]{zhang2021} by solving entire solutions of \eqref{EQ1intro} when the right-hand side is replaced by $b_1(z)e^{p_1(z)}+b_2(z)e^{p_2(z)}$, where $b_1(z)$ and $b_2(z)$ are now nonzero polynomials, $p_1(z)$ and $p_2(z)$ are two nonconstant polynomials of the same degree~$k$ with distinct leading coefficients, respectively; see Theorem~\ref{maintheorem1} in section~\ref{Application to Tumura--Clunie differential equations}. To this end, we first need to give a generalization of \cite[Lemma~2.3]{zhang2021} on first-order differential equations. Consider the differential equation
\begin{equation}\label{EQ0}
f'-\kappa f=\beta,
\end{equation}
where $\kappa$ is a nonzero polynomial and $\beta$ is an entire small function of $e^{p(z)}$, where $p(z)$ is a primitive function of $\kappa$. By elementary integration, the general solution of \eqref{EQ0} is $f=ce^{p(z)}+H(z)$ with
\begin{equation}\label{Hfunction}
H(z)=e^{p(z)}\int_0^{z}\beta(t) e^{- p(t)}dt.
\end{equation}
The function $H(z)$ with $\beta$ being a polynomial has frequently appeared in the study of
differential equations and complex dynamics; see, e.g., \cite{Banklangley1987,Baker1984,Bergweiler1992,Bergweiler1993,Haruta1999,Wolff2021}. Below we describe the growth behavior of $H(z)$ in the complex plane~$\mathbb{C}$.

We write $p(z)=\alpha z^{k}+p_{k-1}(z)$, where $\alpha=a+ib\not=0$ with $a,b$ real and $p_{k-1}(z)$ is a polynomial of degree at most $k-1$. Denote
\begin{equation}\label{expon1}
\delta(p,\theta)=a\cos k\theta-b\sin k\theta, \quad \theta\in[0,2\pi).
\end{equation}
Let $\varepsilon>0$ be given and small. Then on the ray $z=re^{i\theta}$, $r\geq 0$, we have: (1) if $\delta(p,\theta)>0$, then there exists an $r_0=r_0(\theta)$ such that $\log |e^{p(re^{i\theta})}|$ is increasing on $[r_0,\infty)$ and $|e^{p(re^{i\theta})}|\geq e^{(1-\varepsilon)\delta(p,\theta)r^{k}}$ there; (2) if $\delta(p,\theta)<0$, then there exists an $r_0=r_0(\theta)$ such that $\log |e^{p(re^{i\theta})}|$ is decreasing on $[r_0,\infty)$ and $|e^{p(re^{i\theta})}|\leq e^{(1-\varepsilon)\delta(p,\theta)r^{k}}$ there; see \cite{Banklangley1987} or \cite[Lemma~5.14]{Laine1993}. Let $\theta_j\in [0,2\pi)$, $j=1,2, \cdots, 2k$, be such that $\delta(p,\theta_j)=0$. We may suppose that $\theta_1< \pi$ and $\theta_1<\theta_2<\cdots<\theta_{2k}$. Then $\theta_j=\theta_1+(j-1)\pi/k$. Moreover, denoting $\theta_{2k+1}=\theta_1+2\pi$, we have $2k$ sectors $S_j$ defined as
\begin{equation}\label{expon10}
S_j=\left\{re^{i\theta}:\, 0\leq r<\infty, \quad \theta_j< \theta<\theta_{j+1}\right\}, \quad j=1,\cdots,2k.
\end{equation}
Moreover, for any $\epsilon>0$, we denote
\begin{equation}\label{expon10 fujia}
\overline{S}_{j,\epsilon}=\{re^{i\theta}:\, 0\leq r<\infty, \quad \theta_j+\epsilon\leq\theta\leq\theta_{j+1}-\epsilon\}, \quad j=1,\cdots,2k.
\end{equation}
Denote by $\overline{S}_j$ and $\overline{S}_{j,\epsilon}$ the closure of $S_j$ and $S_{j,\epsilon}$, respectively. We have

\begin{theorem}\label{maintheorem0}
Suppose that $p(z)$ is polynomial with degree $k\geq 1$ and that $\beta$ is a small entire function of $e^{p(z)}$.
Then for each $S_j$ where $\delta(p,\theta)>0$, there is a constant $a_j$ such that $|H(re^{i\theta})-a_je^{p(re^{i\theta})}|\leq e^{o(1)r^{k}}$ uniformly as $r\to\infty$ in $\overline{S}_{j,\epsilon}$; for each $S_j$ where $\delta(p,\theta)<0$, $|H(re^{i\theta})-ae^{p(re^{i\theta})}|\leq e^{o(1)r^{k}}$ uniformly as $r\to\infty$ in $\overline{S}_{j}$ for any constant $a$.
\end{theorem}

By the same arguments as in \cite[Lemma~2.3]{zhang2021} together with the growth properties of $e^{p(z)}$ mentioned before, we easily prove Theorem~\ref{maintheorem0}. We omit the proof. Here we give two remarks: First, if $\rho=\max\{\rho(\beta),k-1\}<k$, then the error term $e^{o(1)r^k}$ in Theorem~\ref{maintheorem0} can be replaced by $e^{r^{\eta}}$ for a constant $\eta>\rho$. Second, Theorem~\ref{maintheorem0} easily extends to the case where the function $\beta(z)$ in \eqref{EQ0} is meromorphic and has at most finitely many poles. If the solution $f$ of equation \eqref{EQ0} is meromorphic, then there is a rational function $R(z)$ such that $R(z)\to 0$ as $z\to\infty$ and $h(z)=f(z)-R(z)$ is entire. It follows that $f(z)=h(z)+R(z)$ and $h$ satisfies $h'-\kappa h=\beta-(R'-\kappa R)$ and $\beta_1=\beta-(R'-\kappa R)$ is a small entire function of $e^{p(z)}$.

The rest of this paper is organized in the following way. In section~\ref{Application to Tumura--Clunie differential equations}, we apply Theorem~\ref{maintheorem0} to Tumura--Clunie type differential equations and provide a generalization of~\cite[Theorem~1.1]{zhang2021}. Our results also improve~\cite[Theorem~1]{LiaoYangZhang2013}. Then we use our results to determine all solutions with relatively few zeros of the second-order linear differential equation: $g''-(b_1e^{p_1}+b_2e^{p_2}+b_3)g=0$, where $b_j(z)$, $j=1,2,3$, are polynomials, $p_1(z)$ and $p_2(z)$ are two polynomials of the same degree with distinct leading coefficients. In section~\ref{Application to complex dynamics}, we consider the first-order linear differential equation $f'-R_1(z) e^{q(z)}f=R_2(z)$ with a polynomial $q(z)$ and two rational functions $R_1(z)$ and $R_2(z)$ and show that $q(z)$ must be a constant when $f$ is of finite order. This equation is related to a class of meromorphic functions appearing in complex dynamics.

\section{Tumura--Clunie differential equations}\label{Application to Tumura--Clunie differential equations}

Let $b_1(z)$ and $b_2(z)$ be two nonzero polynomials and $p_1(z)$ and $p_2(z)$ be two polynomials of the same degree~$k\geq 1$ with distinct leading coefficients $\alpha_1$ and $\alpha_2$, respectively. In this section, we solve entire solutions of the Tumura--Clunie type differential equation
\begin{equation}\label{EQ1}
f^n+P(z,f)=b_1e^{p_1}+b_2e^{p_2},
\end{equation}
where $n\geq 2$ and $P(z,f)$ is a differential polynomial in $f$ of degree $\leq n-1$ with meromorphic functions of order less than~$k$ as coefficients. In the following, a differential polynomial in $f$ will always have meromorphic functions of order less than~$k$ as coefficients and thus we will omit mentioning this. By doing a linear transformation $z\to z/\alpha_1^{1/k}$ or $z\to z/\alpha_2^{1/k}$, if necessary, we may suppose that the leading coefficients of $p_1$ and $p_2$ are $\alpha_1=1$ and $\alpha_2=\alpha$, respectively, and $|\alpha|\leq 1$. With these settings, we prove the following

\begin{theorem}\label{maintheorem1}
Let $n\geq2$ be an integer and $P(z,f)$ be a differential polynomial in $f$ of degree $\leq{n-1}$. Let $b_1$, $b_2$ be two nonzero polynomials and $p_1$, $p_2$ be two nonconstant polynomials of degree~$k$ with distinct leading coefficients $1$ and $\alpha$, respectively, and $p_1(0)=p_2(0)=0$. Suppose that \eqref{EQ1} has an entire solution~$f$. Then $\alpha$ is real and rational. Moreover,
\begin{itemize}
  \item [(1)] if $\alpha<0$, then $\alpha=-1$ and $f=\gamma+\gamma_1e^{p_1/n}+\gamma_2e^{p_2/n}$,
              where $\gamma$ is an entire function of order less than $k$ and $\gamma_1$, $\gamma_2$ are
              two polynomials such that $\gamma_1^n=b_1$, $\gamma_2^n=b_2$;
  \item [(2)] if $0<\alpha<1$, letting $m$ be the smallest integer such that $\alpha\leq[(m+1)n-1]/[(m+1)n]$,
              then $f=\gamma+\gamma_1\sum_{j=0}^mc_j(b_2/b_1)^je^{[jn(\alpha-1)+1]p_1/n}$, where $\gamma$ is a function of order less than~$k$ and $\gamma_1$ is a polynomial such that $\gamma_1^n=b_1$ and $c_0$, $\cdots$, $c_m$ are constants such that $c_0^{n}=1$ when $m=0$, and $c_0^{n}=nc_0^{n-1}c_1=1$ when $m=1$, and $c_0^{n}=nc_0^{n-1}c_1=1$ and $\sum_{\substack{j_0+\cdots+j_m=n,\\j_1+\cdots+mj_m=j}}\frac{n!}{j_0!j_1!\cdots j_m!}c_0^{j_0}c_1^{j_1}\cdots c_m^{j_m}=0$, $j=2,\cdots,m$, when $m\geq 2$. Moreover, we have $p_2=\alpha p_1$ when $m\geq 1$.
\end{itemize}
\end{theorem}

%We note that an application of the method in \cite{Ishizalikazuya19971} to equation \eqref{EQ1intro} will yield that \eqref{EQ1intro} cannot have entire solutions when $\alpha_2/\alpha_1$ is not real and thus~\cite[Theorem~1.1]{zhang2021} actually

We now apply Theorem~\ref{maintheorem1} to linear differential equations. In the second-order case, the linear differential equation
\begin{equation}\label{bank-laine}
g''+Ag=0,
\end{equation}
where $A$ is an entire function, has attracted much interest; see \cite[chapter~5]{Laine1993} and references therein. A famous conjecture concerning the zero distribution of solutions of \eqref{bank-laine} is known as the \emph{Bank--Laine conjecture}~\cite{Banklaine1982,Banklaine1982-2}: Let $g_1$ and $g_2$ be two linearly independent solutions of \eqref{bank-laine} and denote by $\lambda(g)$ the exponent of convergence of zeros
of~$g$. Is $\max\{\lambda(g_1),\lambda(g_2)\}=\infty$ whenever $\rho(A)$ is not an integer? Recently, this conjecture was disproved by Bergweiler and Eremenko \cite{Bergweilereremenko2017,Bergweilereremenko2019}. In their construction of the counterexamples, they started from the solutions of \eqref{bank-laine} when $A$ is a polynomial of $e^z$ of degree~2.

Suppose that the coefficient $A$ in \eqref{bank-laine} has the form $A=-[b_1(z)e^{p_1(z)}+b_2(z)e^{p_2(z)}+b_3(z)]$, where $b_1(z)$, $b_2(z)$ and $p_1(z)$, $p_2(z)$ are as in Theorem~\ref{maintheorem1} and $b_3(z)$ is a polynomial. If $\alpha$ is non-real or is real negative, then all non-trivial solutions of \eqref{bank-laine} satisfy $\lambda(f)=\infty$; see \cite{banklainelangely1989,Ishizalikazuya19971}. When $\alpha$ is positive, Ishizaki~\cite[Theorem~1]{Ishizaki19970} proved: If $0<\alpha<1/2$ or if $b_3\equiv0$ and $3/4<\alpha<1$, then all non-trivial solutions of \eqref{bank-laine} satisfy $\lambda(g)\geq k$. We will improve Ishizaki's result by showing that the condition $b_3\equiv0$ can be removed. In fact, with Theorem~\ref{maintheorem1} at our disposal, we are able to determine all solutions such that $\lambda(g)<k$ of equation \eqref{bank-laine} with $A$ above.

\begin{theorem}\label{maintheorem2}
Let $b_1$, $b_2$ and $b_3$ be polynomials such that $b_1b_2\not\equiv0$ and $p_1$, $p_2$ be two polynomials of degree~$k\geq 1$ with distinct leading coefficients $1$ and $\alpha$, respectively, and~$p_1(0)=p_2(0)=0$. Let $A=-(b_1e^{p_1}+b_2e^{p_2}+b_3)$. Suppose that \eqref{bank-laine} has a non-trivial solution such that $\lambda(g)<k$. Then $\alpha=1/2$ or $\alpha=3/4$. Moreover,
\begin{itemize}
\item [(1)]
if $\alpha=1/2$, then $p_2=p_1/2$, $g=\kappa e^{h}$, where $\kappa$ is a polynomial with simple roots only and $h$ satisfies $h'=\gamma_1e^{p_1/2}+\gamma$ with $\gamma_1$ and $\gamma$ being two polynomials such that $\gamma_1^2=b_1$, $2\gamma_1\gamma+\gamma_1'+\gamma_1p_1'/2+2\kappa'/\kappa\gamma_1=b_2$ and $\gamma^2+\gamma'+2\gamma\kappa'/\kappa+\kappa''/\kappa=b_3$;

\item [(2)]
if $\alpha=3/4$, then $p_1=z$, $p_2=3z/4$ and $g=e^{h}$, where $h$ satisfies $h'=-4c^2e^{z/2}+ce^{z/4}-1/8$ and $A=-(16c^2e^{z}-8c^3e^{3z/4}+1/64)$, where $c$ is a nonzero constant.
\end{itemize}
\end{theorem}

\begin{proof}
We write $g=\kappa e^{h}$, where $h$ is an entire function and $\kappa$ is the canonical product from the zeros of $g$ and satisfies $\rho(\kappa)=\lambda(\kappa)<k$. By denoting $f=h'$, from \eqref{bank-laine} we have
\begin{equation}\label{linear eq1}
f^2+f'+2\frac{\kappa'}{\kappa}f+\frac{\kappa''}{\kappa}=b_1(z)e^{p_1(z)}+b_2(z)e^{p_2(z)}+b_3(z).
\end{equation}
By Theorem~\ref{maintheorem1} together with previous discussions we know that $\alpha$ is a positive rational number. Below we consider the two cases where $0<\alpha\leq 1/2$ and $1/2<\alpha<1$, respectively.

When $0<\alpha\leq 1/2$, by Theorem~\ref{maintheorem1} we may write $f=\gamma_1e^{p_1/2}+\gamma$, where $\gamma_1$ is a polynomial such that $\gamma_1^2=b_1$ and $\gamma$ is an entire function of order less than~$k$. Substitution into equation \eqref{linear eq1} gives
\begin{equation}\label{1Eq3.29 prev0}
\begin{split}
2\gamma_1\left(\frac{\kappa'}{\kappa}+\frac{1}{2}\frac{\gamma_1'}{\gamma_1}+\frac{p_1'}{4}+\gamma\right)e^{p_1/2}-b_2e^{p_2}+\gamma^2+\gamma'+2\gamma\frac{\kappa'}{\kappa}+\frac{\kappa''}{\kappa}-b_3=0.
\end{split}
\end{equation}
If $\alpha\not=1/2$, then by Borel's lemma (see {\cite[pp.~69--70]{yangy:03}}) we get $b_2\equiv0$, a contradiction. Therefore, $\alpha=1/2$. Then by rewriting equation \eqref{1Eq3.29 prev0} and applying Borel's lemma to the resulting equation again, we get $\gamma^2+\gamma'+2\gamma\kappa'/\kappa+\kappa''/\kappa=b_3$ and $2\gamma_1\gamma+\gamma_1'+\gamma_1p_1'/2+2\kappa'/\kappa\gamma_1=b_2e^{p_2-p_1/2}$. From the second relation we see that $\kappa$ has only finitely many zeros and thus we may suppose that $\kappa$ is a polynomial. Then from the first relation we see that $\gamma$ is also a polynomial and all zeros of $\kappa$ are simple. This implies that $p_2=p_1/2$.

When $1/2< \alpha<1$, by Theorem~\ref{maintheorem1} we may write $f=\gamma+\gamma_1\sum_{j=0}^mc_j(b_2/b_1)^je^{[2j(\alpha-1)+1]p_1/2}$, where $m\geq 1$ is an integer, $\gamma_1$ is a polynomial such that $\gamma_1^2=b_1$ and $\gamma$ is a meromorphic function with finitely many poles and of order less than~$k$, and $c_0$, $\cdots$, $c_m$ are constants such that $c_0^{2}=2c_0c_1=1$ when $m=1$, and $c_0^{2}=2c_0c_1=1$ and $\sum_{\substack{j_0+\cdots+j_m=2,\\j_1+\cdots+mj_m=j}}\frac{2}{j_0!j_1!\cdots j_m!}c_0^{j_0}c_1^{j_1}\cdots c_m^{j_m}=0$, $j=2,\cdots,m$, when $m\geq 2$. Note that $(2m-1)/(2m)<\alpha\leq [2(m+1)-1]/[2(m+1)]$ and $(m+j+1)(\alpha-1)+1=j(\alpha-1)+1/2$, $j=0,1,\cdots,m-1$, when $\alpha=[2(m+1)-1]/[2(m+1)]$. By substituting this expression together with $p_2=\alpha p_1$ into equation \eqref{linear eq1}, we get
\begin{equation}\label{1Eq3.29 prev}
\begin{split}
&\gamma_1^2\sum_{k_0=m+1}^{2m}C_{k_0}\left(\frac{b_2}{b_1}\right)^{k_0}e^{(k_0\alpha-k_0+1)p_1}+\gamma^2+\gamma'+2\frac{\kappa'}{\kappa}\gamma+\frac{\kappa''}{\kappa}-b_3\\
&+\gamma_1\sum_{j=0}^mc_j\left(\frac{b_2}{b_1}\right)^{j}\left[2\gamma+\frac{\gamma_1'}{\gamma_1}+2\frac{\kappa'}{\kappa}+j\frac{(b_2/b_1)'}{b_2/b_1}+\frac{2j(\alpha-1)+1}{2}p_1'\right]e^{\frac{2j(\alpha-1)+1}{2}p_1}=0,
\end{split}
\end{equation}
where $C_{k_0}=\sum_{\substack{j_0+\cdots+j_m=2,\\j_1+\cdots+mj_m=k_0}}\frac{2!}{j_0!j_1!\cdots j_m!}c_0^{j_0}c_1^{j_1}\cdots c_m^{j_m}$, $k_0=m+1,\cdots,2m$. We may rewrite the left-hand side of equation \eqref{1Eq3.29 prev} by combining the same exponential terms together. By Borel's lemma, all coefficients of the exponential terms in the resulting equation vanish identically. Therefore, we have $\gamma^2+\gamma'+2\gamma\kappa'/\kappa+\kappa''/\kappa=b_3$ and, by looking at the coefficients of the terms $e^{[2m(\alpha-1)+1]p_1/2}$ and $e^{p_1/2}$ in \eqref{1Eq3.29 prev}, respectively, that
\begin{equation}\label{ossica}
\begin{split}
2\gamma+\frac{\gamma_1'}{\gamma_1}+2\frac{\kappa'}{\kappa}+\left[m\frac{(b_2/b_1)'}{b_2/b_1}+\frac{2m(\alpha-1)+1}{2}p_1'\right]&=0,\\
\gamma_1\tilde{C}_{k_0}\left(\frac{b_2}{b_1}\right)^{k_0}+c_0\left(2\gamma+\frac{\gamma_1'}{\gamma_1}+2\frac{\kappa'}{\kappa}+\frac{1}{2}p_1'\right)&=0,
\end{split}
\end{equation}
where $\tilde{C}_{k_0}=C_{m+1}$ when $\alpha=[2(m+1)-1]/[2(m+1)]$ or $\tilde{C}_{k_0}=0$ when $\alpha<[2(m+1)-1]/[2(m+1)]$. Since $\gamma$ has only finitely many poles, we see from the first equation in \eqref{ossica} that $\kappa$ has only finitely many zeros. We may suppose that $\kappa$ is a polynomial. It follows that $\gamma$ is a rational function having simple poles only. From the equations in \eqref{ossica} we get
\begin{equation}\label{ossica12}
\begin{split}
-\tilde{C}_{k_0}\gamma_1(b_2/b_1)^{k_0}=mc_0[(b_2/b_1)'/(b_2/b_1)+(\alpha-1)p_1'],
\end{split}
\end{equation}
which is possible only when $\tilde{C}_{k_0}\not=0$ and thus $\alpha=[2(m+1)-1]/[2(m+1)]$.
Suppose that $z_0$ is a zero of $b_1$ of order~$l_1$. We see that $z_0$ must also be a zero of $b_2$ for otherwise $\gamma_1(b_2/b_1)^{m+1}$ would have a pole at $z_0$ of order $(m+1-1/2)l_1=(m+1/2)l_1>1$, which is impossible. Suppose that $z_0$ is a zero of $b_2$ of order~$l_2$. If $l_1>l_2$, then from equation \eqref{ossica12} we see that $(m+1)(l_1-l_2)-l_1/2\leq 1$, which is impossible since $f$ is entire and we also have $m(l_1-l_2)-l_1/2=1$. Therefore, $l_1\leq l_2$ and we conclude from equation \eqref{ossica12} that $b_2/b_1$ is a constant. Since $f$ is entire, this then yields that $\gamma$ is a polynomial. Then from the second equation of \eqref{ossica} we deduce that $\kappa$, $\gamma_1$ are both constants. Then from equation \eqref{ossica12} we see that $p_1'$ is a constant. Finally, from the first equation in \eqref{ossica} and the equation $\gamma^2+\gamma'+2\gamma\kappa'/\kappa+\kappa''/\kappa=b_3$ we see that $\gamma$ and $b_3$ are also both constants.

Now, since $p_1(0)=p_2(0)=0$, we have $p_1=z$ and $p_2=\alpha z$ and can assume that $f=\gamma+\sum_{j=0}^mc_je^{[2j(\alpha-1)+1]z/2}$, where $\gamma$ and $c_0$, $\cdots$, $c_m$ are constants such that $c_0^{n}=b_1$, $2c_0c_1=b_2$ when $m=1$, and $c_0^{2}=b_1$ and $2c_0c_1=b_2$ and $C_{k_0}=\sum_{\substack{j_0+\cdots+j_m=2,\\j_1+\cdots+mj_m=k_0}}\frac{2!}{j_0!j_1!\cdots j_m!}c_0^{j_0}c_1^{j_1}\cdots c_m^{j_m}=0$, $k_0=2,\cdots,m$, when $m\geq 2$. Suppose that $m\geq 2$. From previous discussions on \eqref{1Eq3.29 prev} we have
\begin{equation}\label{uniqued}
\begin{split}
C_{m+1}+2\gamma c_0+c_0/2&=0,\\
C_{m+2}+2\gamma c_1+c_1[2(\alpha-1)+1]/2&=0,\\
\cdots\\
C_{2m-1}+2\gamma c_{m-2}+c_{m-2}[2(m-2)(\alpha-1)+1]/2&=0,\\
C_{2m}+2\gamma c_{m-1}+c_{m-1}[2(m-1)(\alpha-1)+1]/2&=0,\\
2\gamma c_{m}+c_{m}[2m(\alpha-1)+1]/2&=0.
\end{split}
\end{equation}
From the last equation in \eqref{uniqued} we get $2\gamma=-[2m(\alpha-1)+1]/2$. By substituting this relation into the first~$m$ equations in \eqref{uniqued} we obtain
\begin{equation}\label{uniqued0}
\begin{split}
2c_{1}c_{m}+2c_2c_{m-1}+\cdots+mc_0&=0,\\
2c_{2}c_{m}+2c_3c_{m-1}+\cdots+(m-1)(\alpha-1)c_1&=0,\\
\cdots\\
2c_{m}c_{m-2}+c_{m-1}^2+3(\alpha-1)c_{m-3}&=0,\\
2c_{m-1}c_m+2(\alpha-1)c_{m-2}&=0,\\
c_m^2+(\alpha-1)c_{m-1}&=0.
\end{split}
\end{equation}
From the last two equations in \eqref{uniqued0} we obtain $c_{m}/c_{m-1}=c_{m-1}/c_{m-2}$. By using this relation we obtain from the last three equations in \eqref{uniqued0} that $c_{m-1}/c_{m-2}=c_{m-2}/c_{m-3}$. By induction we finally obtain that $c_{m}/c_{m-1}=c_{m-1}/c_{m-2}=\cdots=c_{1}/c_{0}$. Denote $t=c_{1}/c_{0}$. Then we have $c_{1}=tc_{0}$ and $c_{2}=t^2c_{0}$. But then from the equation $C_2=2c_0c_2+c_1^2=0$ we get $3t^2c_0^2=0$, a contradiction. Therefore, we must have $m=1$. Then we have the first and last equations in \eqref{uniqued} for $m=1$, i.e., $c_1^2+2\gamma c_0+c_0/2=0$ and $2\gamma c_1+c_1/4=0$, which together with the  equations $\gamma^2=b_3$, $c_0^2=b_1$ and $2c_0c_1=b_2$ give $\gamma=-1/8$ and $c_0=-4c_1^2$, $b_1=16c_1^2$, $b_2=-8c_1^3$ and $b_3=1/64$. We complete the proof.

\end{proof}

To prove Theorem~\ref{maintheorem1}, we first introduce the definition of \emph{$R$--set}: An $R$--set in the complex plane is a countable union of discs whose radii have finite sum. Then the set of angles $\theta$ for which the ray $z=re^{i\theta}$, $\theta\in [0,2\pi)$, meets infinitely many discs of a given $R$--set has linear measure zero; see \cite{Banklangley1987} or \cite[p.~84]{Laine1993}. Note that any finite union of $R$--sets is still an $R$--set. Let $f(z)$ be an entire solution of \eqref{EQ1}. We denote the union of all $R$--sets associated with $f(z)$ and each coefficient of $P(z,f)$ by $\hat{R}$ from now on. In the proof of Theorem~\ref{maintheorem1}, after taking the derivatives on both sides of equation \eqref{EQ1}, there may be some new coefficients appearing in the resulting equation. We will always assume that $\hat{R}$ also contains those $R$-sets associated with these new coefficients. Recall the following pointwise estimate due to Gundersen~\cite{gundersen:88}.

\begin{lemma}[\cite{gundersen:88}]\label{derivativelemma}
Let $f$ be a transcendental meromorphic function of finite order $\rho(f)$, and let $\varepsilon>0$ be a given constant. Then there exists a set $E\in [0,2\pi)$ that has linear measure zero, such that if $\psi_0\in[0,2\pi)-E$, then there exists a constant $r_0=r_0(\psi_0)>1$ such that for all $z$ satisfying $\arg z=\psi_0$ and $|z|\geq r_0$, for all positive integers $k>j\geq 0$, we have
\begin{equation*}
\left|\frac{f^{(k)}(z)}{f^{(j)}(z)}\right|\leq |z|^{(k-j)(\rho(f)-1+\varepsilon)}.
\end{equation*}
\end{lemma}

Denote the maximal one of the orders of the coefficients of $P(z,f)$ by $\rho_0$. From now on we let $\rho$ be a fixed constant satisfying $\max\{\rho_0,k-1\}<\rho<k$. Since $\alpha\not=1$, then by Steinmetz's result~\cite{Steinmetz1978} for exponential polynomials, we have $T(r,b_1e^{p_1}+b_2e^{p_2})=C(1+o(1))r^k$, $r\to\infty$, for some nonzero constant $C$ depending only on $\alpha$ and that $m(r,1/(b_1e^{p_1}+b_2e^{p_2}+\gamma))=o(1)r^k$, $r\to\infty$, for any nonzero entire function $\gamma$ of order less than~$k$. Then by slightly modifying the proof of \cite[Lemma~2.4]{zhang2021}, we easily obtain the following

\begin{lemma}\label{growthlemmapre}
Under the assumptions of Theorem~\ref{maintheorem1}, we have $\rho(f)=k$ and, after a possible linear transformation $f\rightarrow f+c$ with a suitable constant $c$, that $m(r,1/f)=O(r^{\rho})$, $m(r,e^{p_1}/f^n)=O(r^{\rho})$ and $m(r,e^{p_2}/f^n)=O(r^{\rho})$. Moreover, if $\alpha<0$, then we have $m(r,e^{p_1+p_2}/f^{2n-1})=O(r^{\rho})$ and $m(r,e^{p_1+p_2}/f^{2n-2})=O(r^{\rho})$; if $0<\alpha<1$, letting $m$ be the smallest integer such that $\alpha\leq [(m+1)n-1]/[(m+1)n]$, then $m(r,e^{(m+1)p_2}/f^{(m+1)n-1})=O(r^{\rho})$.
\end{lemma}

We write $\alpha=a+ib$, where $a$ and $b$ are both real. By slightly modifying the proof
in \cite[Theorem~3.2]{Ishizalikazuya19971}, we can prove the following Lemmas~\ref{growthlemma} and~\ref{reallemma}, which show that $\alpha$ must be real under the assumptions of Theorem~\ref{maintheorem1}.

\begin{lemma}\label{growthlemma}
Under the assumptions of Theorem~\ref{maintheorem1}, let $\theta\in[0,2\pi)$ be such that the ray $z=re^{i\theta}$ meets only finitely many discs in $\hat{R}$. If $\delta(p_1,\theta)>0$ and $\delta(p_1,\theta)>\delta(p_2,\theta)$, then $f(re^{i\theta})^n= b_1(re^{i\theta})(1+o(1))e^{p_1(re^{i\theta})}$, $r\to\infty$; if $\delta(p_2,\theta)>0$ and $\delta(p_2,\theta)>\delta(p_1,\theta)$, then $f(re^{i\theta})^n= b_2(re^{i\theta})(1+o(1))e^{p_2(re^{i\theta})}$, $r\to\infty$.
\end{lemma}

\begin{proof}
Let $\varepsilon>0$ be given and small. By Lemmas~\ref{derivativelemma} and~\ref{growthlemmapre}, there exists a constant $r_0=r_0(\theta)>1$ such that for all $z$ satisfying $|z|=r\geq r_0$, the ray $z=re^{i\theta}$ does not meet $\hat{R}$, and for all positive integers $j$,
\begin{equation}\label{lemmagrw0}
\left|\frac{f^{(j)}(re^{i\theta})}{f(re^{i\theta})}\right|\leq r^{j(k-1+\varepsilon)}.
\end{equation}
For each coefficient of $P(z,f)$, say $a_l$, we have $\rho(a_l)+\varepsilon\leq \rho<k$ and thus
\begin{equation}\label{lemmagrw1}
|a_l(re^{i\theta})|\leq e^{r^{\rho}}
\end{equation}
for all sufficiently large $r$. Consider first the case that $\delta(p_1,\theta)>0$ and $\delta(p_1,\theta)>\delta(p_2,\theta)$. Then along the ray $z=re^{i\theta}$ we have $b_1(re^{i\theta})e^{p_1(re^{i\theta})}+b_2(re^{i\theta})e^{p_2(re^{i\theta})}=b_1(re^{i\theta})(1+o(1))e^{p_1(re^{i\theta})}$, $r\to\infty$.

Recalling that $P(z,f)=\sum^{m}_{l=1}a_{l}f^{n_{l0}}(f')^{n_{l1}}\cdots(f^{(s)})^{n_{ls}}$, where $m$ is an integer and $n_{l0}+n_{l1}+\cdots n_{ls}\leq n-1$, we may write
\begin{equation}\label{lemmaQ10}
P(z,f)=\sum^{m}_{l=1}\hat{a}_{l}f^{n_{l0}+n_{l1}+\cdots+n_{ls}}
\end{equation}
with the new coefficients $\hat{a}_l=a_l(f'/f)^{n_{l1}}\cdots(f^{(s)}/f)^{n_{ls}}$. Denote by $L$ the greatest order of the derivatives of $f$ in $P(z,f)$. Suppose that there is an infinite sequence $z_j=r_je^{i\theta}$ such that $|f(r_je^{i\theta})|\leq e^{o(1)r_j^k}$. Then, from \eqref{EQ1}, \eqref{lemmagrw0}, \eqref{lemmagrw1} and \eqref{lemmaQ10} we have
\begin{equation}\label{lemmaQ3}
|b_1(r_je^{i\theta})(1+o(1))e^{p_1(r_je^{i\theta})}|\leq |f(r_je^{i\theta})^n+P(r_je^{i\theta},f)|\leq e^{o(r_j^k)}+mr_j^{L(k-1+\varepsilon)}e^{r_j^{\rho}}e^{o(r_j^k)},
\end{equation}
which is impossible when $r_j$ is large. Therefore, along the ray $z=re^{i\theta}$, there must be some $\eta>0$ such that $|f(re^{i\theta})|\geq e^{\eta r^{k}}$ for all sufficiently large $r$. We write \eqref{EQ1} as
\begin{equation}\label{lemmaQ4}
f(re^{i\theta})^n\left[1+\frac{P(re^{i\theta},f)}{f(re^{i\theta})^n}\right]=b_1(re^{i\theta})e^{p_1(re^{i\theta})}+b_2(re^{i\theta})e^{p_2(re^{i\theta})}.
\end{equation}
Then by combining \eqref{EQ1}, \eqref{lemmagrw0}, \eqref{lemmagrw1} and \eqref{lemmaQ10}, we easily see that $P(re^{i\theta},f)/f(re^{i\theta})^n\to 0$ as $r\to\infty$. Thus our first assertion follows. The second assertion can be proved in the same way.

\end{proof}

\begin{lemma}\label{reallemma}
Under the assumptions of Theorem~\ref{maintheorem1}, we have that $\alpha$ is real.
\end{lemma}

\begin{proof}
For simplicity, we denote $P=P(z,f)$. By taking the derivatives on both sides of \eqref{EQ1}, we have
\begin{equation}\label{1Eq2.1}
nf^{n-1}f'+P'=b_1B_1e^{p_1}+b_2B_2e^{p_2},
\end{equation}
where $B_1=b_1'/b_1+p_1'$ and $B_2=b_2'/b_2+p_2'$. Obviously, $B_1B_2\not\equiv0$. By eliminating $e^{p_2}$ and $e^{p_1}$ from \eqref{EQ1} and \eqref{1Eq2.1}, respectively, we get
\begin{equation}\label{1Eq2.2}
b_2B_2f^n-nb_2f^{n-1}f'+b_2B_2 P-b_2P'=A_1{e^{p_1}},
\end{equation}
and
\begin{equation}\label{1Eq3.1}
b_1B_1f^n-nb_1f^{n-1}f'+b_1B_1 P-b_1P'=-A_1{e^{p_2}},
\end{equation}
where $A_1=b_1b_2(B_2-B_1)$. We see that $A_1\not\equiv0$, for otherwise we have $B_2=B_1$ and by integration we have $b_2e^{p_2}=Cb_1e^{p_1}$ for some nonzero constant $C$, which is impossible since $\alpha\not=1$. Denote $F_1=f'-(B_1/n)f$ and $F_2=f'-(B_2/n)f$.

Suppose that $\alpha$ is not real. Recall from \eqref{expon1} that $\delta(p_1,\theta)=\cos k\theta$ and $\delta(p_2,\theta)=a\cos k\theta-b\sin k\theta$. Denote by $S_{1,i}$ and $S_{2,j}$, $i,j=1,\cdots,2k$, the corresponding sectors for $p_1$ and $p_2$ defined in \eqref{expon10}, respectively. Then the sector $S_{1,i}$ where $\delta(p_1,\theta)>0$ intersects with one sector $S_{2,j}$ where $\delta(p_2,\theta)>0$ and also with one sector $S_{2,j}$ where $\delta(p_2,\theta)<0$.

Let $\theta\in[0,2\pi)$ be such that the ray $z=re^{i\theta}$ meets only finitely discs in $\hat{R}$. We first consider the growth of $F_1$ along the ray $z=re^{i\theta}$. From \eqref{1Eq2.2} we have
\begin{equation}\label{1Eq2.2fy0}
|F_1|\leq \left|\frac{B_2 P}{nf^{n-1}}\right|+\left|\frac{P'}{nf^{n-1}}\right|+\left|\frac{A_2{e^{p_1}}}{nb_2f^{n-1}}\right|.
\end{equation}
Note that $P$ and $P'$ are two differential polynomials in $f$ of degree~$\leq n-1$.
If along the ray $z=re^{i\theta}$ we have $\delta(p_1,\theta)<0$, then by using the method in the proof of \cite[Theorem~2.3]{Ishizalikazuya19971} together with Lemma~\ref{derivativelemma}, we may consider the two cases where $|f(re^{i\theta})|\leq 1$ and $|f(re^{i\theta})|> 1$, respectively, and show that
\begin{equation}\label{1Eq2.2fy1}
\left|F_1(re^{i\theta})\right|\leq e^{r^{\rho}}
\end{equation}
for all sufficiently large $r$. Let $\epsilon>0$ be given and small. If along the ray $z=re^{i\theta}$ we have $\delta(p_1,\theta)>0$ and $\delta(p_1,\theta)>\delta(p_2,\theta)$, then by Lemmas~\ref{derivativelemma} and~\ref{growthlemma}, we have from \eqref{1Eq2.2fy0} that
\begin{equation}\label{1Eq2.2fy2}
\left|F_1(re^{i\theta})\right|\leq e^{[\delta(p_1,\theta)/n]r^{k+\epsilon}}
\end{equation}
for all sufficiently large $r$. If along the ray $z=re^{i\theta}$ we have $\delta(p_2,\theta)>0$ and $\delta(p_2,\theta)>\delta(p_1,\theta)>(n-1)\delta(p_2,\theta)/n$, then by
Lemmas~\ref{derivativelemma} and~\ref{growthlemma}, we have from \eqref{1Eq2.2fy0} that
\begin{equation}\label{1Eq2.2fy3}
\left|F_1(re^{i\theta})\right|\leq e^{[\delta(p_1,\theta)-(n-1)\delta(p_2,\theta)/n]r^{k+\epsilon}}
\end{equation}
for all sufficiently large $r$. Further, if along the ray $z=re^{i\theta}$ we have $\delta(p_2,\theta)>0$ and $0<\delta(p_1,\theta)<(n-1)\delta(p_2,\theta)/n$, then by Lemmas~\ref{derivativelemma} and~\ref{growthlemma}, we have from \eqref{1Eq2.2fy0} that
\begin{equation}\label{1Eq2.2fy4}
\left|F_1(re^{i\theta})\right|\leq e^{r^{\rho}}
\end{equation}
for all sufficiently large $r$. Note that $F_1$ has only finitely many poles. With the estimates in \eqref{1Eq2.2fy1}--\eqref{1Eq2.2fy4}, we may apply Phragm\'{e}n--Lindel\"{o}f to $F_1$ as in the proof of \cite[Theorem~2.3]{Ishizalikazuya19971} and show that there is an $\rho_1$ such that $\rho<\rho_1<k$ and $|F_1(re^{i\theta})|\leq e^{r^{\rho_1}}$ for all $\theta$. Thus $\sigma(F_1)\leq \rho_1$. Here we omit the details. Similarly, we also have $\sigma(F_2)\leq \rho_1$. Since $f=\frac{n}{B_2-B_1}(F_1-F_2)$, then by comparing the orders on both sides, we get $\sigma(f)<k$. This is a contradiction to Lemma~\ref{growthlemmapre}. Therefore, $\alpha$ must be real.

\end{proof}

%Finally, when $\alpha$ is real, by looking at the proof of \cite[Lemma~2.4]{zhang2021}, we easily get the following

%\begin{lemma}\label{smallcoelemma}
%Let $m\geq0$ be an integer and $t_0,t_1,\cdots, t_m$ be positive numbers such that $t_0>t_1>\cdots>t_m$. Let $\chi_i$ be rational functions. Suppose that $f=\sum_{i=0}^m\chi_ie^{t_ip_1}+H_0$ is an entire solution of \eqref{EQ1}, where $|H_0|=O(e^{r^{\rho}})$ uniformly as $z\to\infty$ in $\sum_{j=0}^{2k}\overline{S}_{j,\epsilon}$ and $\overline{S}_{j,\epsilon}$ is defined in \eqref{expon10 fujia} for $p_1$. Then $H_0(z)$ is a small function of $f$.
%\end{lemma}

Now we begin to prove Theorem~\ref{maintheorem1}.

\renewcommand{\proofname}{Proof of Theorem~\ref{maintheorem1}.}

\begin{proof}

By Lemma~\ref{reallemma}, $\alpha$ is a nonzero real number such that $|\alpha|\leq1$. Below we consider the two cases where $-1\leq \alpha< 0$ and $0<\alpha<1$, respectively.

\vskip 4pt

\noindent\textbf{Case~1:} $-1\leq \alpha< 0$.

\vskip 4pt

By differentiating on both sides of \eqref{1Eq2.2} and then eliminating $e^{p_1}$ from \eqref{1Eq2.2} and the resulting equation, we get
\begin{equation}\label{1Eq2.3}
h_1 f^n+h_2f^{n-1}f'+h_3f^{n-2}(f')^2+h_4f^{n-1}f''+P_1=0,
\end{equation}
where
\begin{equation*}
P_1=(A_1'+p_1'A_1)(b_2B_2P-b_2P')-A_1(b_2B_2P-b_2P')'
\end{equation*}
is a differential polynomial in $f$ of degree $\leq n-1$, and
\begin{equation*}
\begin{split}
h_1&=b_2B_2(A_1'+p_1'A_1)-(b_2B_2)'A_1,\\
h_2&=-n b_2A_1(p_1'+p_2')-nb_2A_1',\\
h_3&=n(n-1)b_2A_1,\\
h_4&=nb_2A_1.
\end{split}
\end{equation*}
By multiplying both sides of equations \eqref{1Eq2.2} and \eqref{1Eq3.1} we have
\begin{equation}\label{1Eq3.2}
g_1 f^{2n}+g_2f^{2n-1}f'+g_3f^{2n-2}(f')^2+P_2=-A_1^2e^{p_1+p_2},
\end{equation}
where
\begin{equation*}
\begin{split}
P_2=&\, b_1b_2(B_2f^n-nf^{n-1}f')(B_1P-P')+b_1b_2(B_1P-P')(B_2P-P')\\
&+b_1b_2(B_1f^n-nf^{n-1}f')(B_2P-P')
\end{split}
\end{equation*}
is a differential polynomial in $f$ of degree $\leq 2n-1$, and
\begin{equation*}
\begin{split}
g_1&=b_1b_2B_1B_2,\\
g_2&=-nb_1b_2(B_1+B_2),\\
g_3&=n^2b_1b_2.
\end{split}
\end{equation*}
By eliminating $(f')^2$ from \eqref{1Eq2.3} and \eqref{1Eq3.2}, we get
\begin{equation}\label{1Eq3.3}
\begin{split}
f^{2n-1}\left[(g_3h_1-h_3g_1)f+(g_3h_2-h_3g_2)f'+g_3h_4f''\right]+P_3=h_3A_1^2e^{p_1+p_2},
\end{split}
\end{equation}
where $P_3=g_3f^nP_1-h_3P_2$ is a differential polynomial in $f$ of degree $\leq 2n-1$. For simplicity, denote
\begin{equation}\label{1Eq3.9a10a}
\varphi=\frac{h_3A_1^2}{g_3h_4}\frac{e^{p_1+p_2}}{f^{2n-1}}-\frac{1}{g_3h_4}\frac{P_3}{f^{2n-1}}.
\end{equation}
Then from equation \eqref{1Eq3.3} we have
\begin{equation}\label{1Eq3.9a1}
f''+H_1f'+H_2f=\varphi,
\end{equation}
where, recalling that $B_1=b_1'/b_1+p_1'$ and $B_2=b_2'/b_2+p_2'$,
\begin{equation}\label{1Eq3.9a1fu1}
\begin{split}
H_1&=\frac{h_2}{h_4}-\frac{g_2 h_3}{g_3h_4}=-\left[\frac{1}{n}(p_1'+p_2')-\frac{n-1}{n}\left(\frac{b_1'}{b_1}+\frac{b_2'}{b_2}\right)+\frac{A_1'}{A_1}\right],\\
H_2&=\frac{h_1}{h_4}-\frac{g_1 h_3}{g_3h_4}=\frac{1}{n}\left[B_2\left(\frac{A_1'}{A_1}-\frac{b_1'}{b_1}\right)-\frac{(b_2B_2)'}{b_2}\right]+\frac{1}{n^2}B_1B_2.
\end{split}
\end{equation}
Note that $\varphi$ has only finitely many poles. 
%For each $r>0$, let $z=re^{i\theta}$. We divide the interval $[0,2\pi)$ into two disjoint sets: $I_1=\{\theta\in[0,2\pi)|:\, |e^{z^k}|\leq1\}$ and $I_2=\{\theta\in[0,2\pi)|:\, |e^{z^k}|>1\}$. Then by using the method in the proof of \cite[Theorem~1.1]{zhang2021} we can show that $m(r,e^{p_1+p_2}/f^{2n-1})=O(r^{\rho})$. Together with this estimate and 
By Lemma~\ref{growthlemmapre} and the lemma on the logarithmic derivative, we obtain from \eqref{1Eq3.9a10a} that $m(r,\varphi)=O(r^{\rho})$ and hence $T(r,\varphi)=O(r^{\rho})$, i.e., $\varphi$ is a function of order $\leq \rho$.

By substituting $f''=-H_1f'-H_2f+\varphi$ into \eqref{1Eq2.3} we have
\begin{equation}\label{1Eq3.3zuih0}
f^{n-2}(f')^2=-\frac{g_1}{g_3}f^n-\frac{g_2}{g_3}f^{n-1}f'+\varphi f^{n-1}-\frac{1}{h_3}P_1.
\end{equation}
By taking the derivatives on both sides of \eqref{1Eq3.3zuih0} and substituting $f''=-H_1f'-H_2f+\varphi$ into the resulting equation inductively, we finally obtain
\begin{equation}\label{1Eq3.3zuih1}
f^{n-j}(f')^j=h_{j,1}f^n+h_{j,2}f^{n-1}f'+P_{1,j-1},  \quad 2\leq j\leq n,
\end{equation}
where $h_{j,1}$ and $h_{j,2}$ are functions of order~$\leq\rho$ and $P_{1,j-1}$ is a differential polynomial in $f$ with degree at most $n-1$. Recall that $F_1=f'-(B_1/n)f$ and $F_2=f'-(B_2/n)f$. Then we can rewrite \eqref{1Eq3.2} as
\begin{equation}\label{1Eq3.3zuih2}
f^{2n-2}F_1F_2+\hat{P}_2=-\frac{A_1^2e^{p_1+p_2}}{n^2b_1b_2},
\end{equation}
where $\hat{P}_2=P_2/(n^2b_1b_2)$. By substituting $f''=-H_1f'-H_2f+\varphi$ and its derivatives into $\hat{P}_2$, we can write $\hat{P}_2$ as $\hat{P}_2=\sum_{j=0}^{n-1}a_jf^{2n-1-j}(f')^{j}+\hat{P}_3$, where $a_j$ are functions of order~$\leq\rho$ and $\hat{P}_3$ is a differential polynomial in $f$ with degree at most $2n-2$. Then by substituting the expressions in \eqref{1Eq3.3zuih1} into $P_2$ we can rewrite $\hat{P}_2$ as $\hat{P}_2=\kappa_1f^{2n-1}+\kappa_2f^{2n-2}f'+\hat{P}_4$, where $\kappa_1$ and $\kappa_2$ are functions of order~$\leq\rho$ and $\hat{P}_4$ is a differential polynomial in $f$ with degree at most $2n-2$. Denote $\phi=F_1F_2+\kappa_1f+\kappa_2f'$. 
%Since $-1\leq \alpha<0$ and $n\geq 2$, then by using the method in the proof of \cite[Theorem~1.1]{zhang2021}, we can also show that 
By dividing $f^{2n-2}$ on both sides of \eqref{1Eq3.3zuih2} and using Lemma~\ref{growthlemmapre} together with the lemma on the logarithmic derivative, we obtain that $m(r,\phi)=O(r^{\rho})$ and so $\phi$ is a function of order~$\leq\rho$. From the expressions of $F_1$ and $F_2$, we easily obtain
\begin{equation}\label{1Eq3.3zuih3jh}
\begin{split}
f=\frac{n}{B_2-B_1}(F_1-F_2), \quad f'=\frac{B_2}{B_2-B_1}F_1-\frac{B_1}{B_2-B_1}F_2.
\end{split}
\end{equation}
By substituting the above two expressions into the equation $\phi=F_1F_2+\gamma_1f+\gamma_2f'$ we get
\begin{equation}\label{1Eq3.3zuih4}
F_1F_2+\mu F_1+\nu F_2=\phi,
\end{equation}
where
\begin{equation*}
\begin{split}
\mu=\frac{n\kappa_1+B_2\kappa_2}{n(B_2-B_1)}, \quad
\nu=\frac{n\kappa_1+B_1\kappa_2}{n(B_1-B_2)}.
\end{split}
\end{equation*}
Obviously, $\mu,\nu$ are both functions of order~$\leq\rho$. We claim that $\phi+\mu\nu\not\equiv0$. Otherwise, from \eqref{1Eq3.3zuih4} we have $(F_1+\nu)(F_2+\mu)\equiv0$, which implies $F_1+\nu\equiv0$ or $F_2+\mu\equiv0$. Suppose $F_1\equiv-\nu$, for example. Then $F_2=[(B_1-B_2)/n]f-\nu$ and it follows by substituting this equation into \eqref{1Eq2.2} and recalling $A_1=b_1b_2(B_2-B_1)$ that
\begin{equation*}
f^{n}-\frac{n\nu}{B_1-B_2}f^{n-1}-\frac{1}{B_1-B_2}(B_2 P-P')=b_1{e^{p_1}}.
\end{equation*}
By Hayman's generalization of the theorem of Tumura--Clunie, there is a function $\gamma$ of order~$\leq\rho$ such that $(f-\gamma)^n=b_1e^{p_1}$. Thus we have $nT(r,f)=T(r,e^{p_1})+O(r^{\rho})$; however, by Steinmetz's result~\cite{Steinmetz1978} we have from equation \eqref{EQ1} that $nT(r,f)+O(r^{\rho})=(1-\alpha)(1+o(1))T(r,e^{p_1})$, which yields that $T(r,e^{p_1})=O(r^{\rho})$, a contradiction. Therefore, $\phi+\mu\nu\not\equiv0$ and also $(F_1+\nu)(F_2+\mu)\not\equiv0$. We write $F_2=(\phi-\mu F_1)/(F_1+\nu)$ and it follows that $T(r,F_2)=T(r,F_1)+O(r^{\rho})$. Then from the first equation in \eqref{1Eq3.3zuih3jh}, we have
\begin{equation*}
T(r,f)\leq{T(r,F_1)+T(r,F_2)+O(r^{\rho})}\leq{2T(r,F_1)+O(r^{\rho})}.
\end{equation*}
On the other hand, by the lemma on the logarithmic derivative, we also have
\begin{equation*}
\begin{split}
m(r,F_1)\leq m\left(r,f\left(\frac{f'}{f}-\frac{1}{n}B_1\right)\right)\leq m(r,f)+O(r^{\rho}),
\end{split}
\end{equation*}
and, since $f$ is transcendental entire, we have $T(r,F_1)\leq T(r,f)+O(r^{\rho})$. Thus $F_1$ and $F_2$ are both functions of order $k$. Recall that $B_1=b_1'/b_1+p_1'$ and $B_2=b_2'/b_2+p_2'$. By simple computations, we may rewrite \eqref{1Eq3.9a1} in the following way:
\begin{equation}\label{1Eq3.3zuih5}
\begin{split}
%F_1'&-\left(\frac{1}{n}p_2'-\frac{b_1'}{b_1}-\frac{n-1}{n}\frac{b_2'}{b_2}+\frac{A_1'}{A_1}\right)F_1=\varphi,\\
F_2'-\left(\frac{1}{n}p_1'-\frac{b_2'}{b_2}-\frac{n-1}{n}\frac{b_1'}{b_1}+\frac{A_1'}{A_1}\right)F_2=\varphi.
\end{split}
\end{equation}
Substituting $F_2=(\phi-\mu F_1)/(F_1+\nu)$ into equation \eqref{1Eq3.3zuih5} gives
\begin{equation*}
\begin{split}
\left[-\mu'+\left(\frac{1}{n}p_1'-\frac{b_2'}{b_2}-\frac{n-1}{n}\frac{b_1'}{b_1}+\frac{A_1'}{A_1}\right)\mu+\varphi\right]F_1^2=P(z,F_1),
\end{split}
\end{equation*}
where $P(z,F_1)$ is a linear polynomial in $F_1$ and $F_1'$ with coefficients formulated in terms of $\mu,\nu,\varphi,\phi$ and their derivatives. For simplicity, denote $\xi=p_1'/n-b_2'/b_2-(n-1)b_1'/nb_1+A_1'/A_1$ and $\psi=-\mu'+\xi\mu+\varphi$. Then $\psi$ is a function of order $\leq \rho$. We claim $\psi\equiv0$.
Otherwise, by Clunie's lemma (see \cite{Banklaine1977,Clunie1962meromorphic}; or \cite{Hayman1964Meromorphic}), we obtain from the above equation that $m(r,{\psi}F_1)=O(r^{\rho})$. But then $T(r,F_1)=m(r,F_1)=m(r,{\psi}F_1/\psi)\leq{m(r,1/\psi)+m(r,{\psi}F_1)}=O(r^{\rho})$,
which is absurd. Hence $\psi\equiv0$. By substituting
$\varphi=\mu'-\xi\mu$ into the second equation of \eqref{1Eq3.3zuih5}, we get $(F_2-\mu)'=\xi(F_2-\mu)$. We see that $b_1$ must be an $n$-square of some polynomial. By integration, we have $F_2-\mu=-C_1A_1b_1^{1/n}/b_1b_2e^{p_1/n}$ for some nonzero
constant $C_1$. By similar arguments, we also have that $b_2$ is a $n$-square of some polynomial and $F_1-\nu=C_2A_1b_2^{1/n}/b_2b_1e^{p_2/n}$ for some nonzero constant $C_2$. Together with $A_1=b_1b_2(B_2-B_1)$, we have from the first equation of \eqref{1Eq3.3zuih3jh} that $f=c_1b_1^{1/n}e^{p_1/n}+c_2b_2^{1/n}e^{p_2/n}+\nu-\mu$. By substituting this solution into \eqref{EQ1} and applying Borel's lemma to the resulting equation as in the proof of \cite[Theorem~1.1]{zhang2021}, we easily obtain that $\alpha=-1$ and $c_1^n=c_2^n=1$. We omit those details.

\vskip 4pt

\noindent\textbf{Case~2:} $0<\alpha<1$.

\vskip 4pt

In this case, we let $m$ be the smallest integer such that $\alpha \leq[(m+1)n-1]/[(m+1)n]$ and  ${\iota_0,\cdots,\iota_m}$ be a finite sequence of functions such that
\begin{equation}\label{1Eq3.27a prepare0}
\begin{split}
\iota_0&=\frac{A_1}{nb_1},\\
\iota_j&=(-1)^{j}\left(\frac{A_1}{nb_1}\right)^{j+1}(jn-1)\cdots(n-1), \quad j=1,2,\cdots,m.
\end{split}
\end{equation}
Recall that $B_1=b_1'/b_1+p_1'$. We also let ${\kappa_0,\cdots,\kappa_m}$ be a finite sequence of functions defined in the following way:
\begin{equation}\label{1Eq3.27a prepare}
\begin{split}
\kappa_0&=\frac{1}{n}\frac{b_1'}{b_1}+\frac{1}{n}p_1',\\
\kappa_j&=\frac{\iota_{j-1}'}{\iota_{j-1}}-\frac{jn-1}{n}\frac{b_1'}{b_1}+\left[j(\alpha-1)+\frac{1}{n}\right]p_1', \quad j=1,2,\cdots,m.
\end{split}
\end{equation}
Then we define $m+1$ functions $G_0$, $G_1$, $\cdots$, $G_m$ in the way that $G_0=f'-\kappa_0f$, $G_{1}=G_{0}'-\kappa_1G_{0}$, $\cdots$, $G_{m}=G_{m-1}'-\kappa_mG_{m-1}$. Now we have equation \eqref{1Eq3.1} and it follows that
\begin{equation}\label{1Eq3.27a}
G_0=f'-\kappa_0f=\iota_0\frac{{e^{p_2}}}{f^{n-1}}+W_0,
\end{equation}
where $W_0=-(B_1P-P')/(nf^{n-1})$. Moreover, when $m\geq 1$, by simple computations we obtain
\begin{equation*}
\begin{split}
G_1=G_0'-\kappa_1G_0=\iota_1\frac{e^{2p_2}}{f^{2n-1}}+W_1,
\end{split}
\end{equation*}
where
\begin{equation*}
W_1=W'_0-\kappa_1W_0-(n-1)\iota_0\frac{e^{p_2}}{f^n}W_0,
\end{equation*}
and by induction that
\begin{equation}\label{1Eq3.27}
\begin{split}
G_{j}=G_{j-1}'-\kappa_jG_{j-1}=\iota_j\frac{e^{(j+1)p_2}}{f^{(j+1)n-1}}+W_{j}, \quad j=1,\cdots,m,
\end{split}
\end{equation}
where
\begin{equation}\label{1Eq3.27a13}
W_{j}=W'_{j-1}-\kappa_jW_{j-1}-(jn-1)\iota_{j-1}\frac{e^{jp_2}}{f^{jn}}W_0, \quad j=1,\cdots,m.
\end{equation}
For an integer $l\geq0$, by elementary computations it is easy to show that $W_0^{(l)}=W_{0l}/f^{n+l-1}$, where $W_{0l}=W_{0l}(z,f)$ is a differential polynomial in $f$ of degree $\leq n+l-1$, and also that $(e^{p_2}/f^{n})^{(l)}=e^{p_2}W_{1l}/f^{n+l}$, where $W_{1l}=W_{1l}(z,f)$ is a differential polynomial in $f$ of degree $\leq n+l$. We see that $W_j$, $1\leq j\leq m$, is formulated in terms of $W_0$ and $e^{p_2}/f^{n}$ and their derivatives. We may write
\begin{equation}\label{1Eq3.28a1}
G_m=\iota_m\frac{e^{(m+1)p_2}}{f^{(m+1)n-1}}+F(W_0,e^{p_2}/f^{n}),
\end{equation}
where $F(W_0,e^{p_2}/f^{n})$ is a combination of $W_0$ and $e^{p_2}/f^{n}$ and their derivatives with functions of order~$\leq \rho$ as coefficients. Note that $G_m$ has only finitely many poles. 
%For each $r>0$, let $z=re^{i\theta}$. We divide the interval $[0,2\pi)$ into two disjoint sets: $I_1=\{\theta\in[0,2\pi)|:\, |e^{z^k}|\leq1\}$ and $I_2=\{\theta\in[0,2\pi)|:\, |e^{z^k}|>1\}$. Then by using the method in the proof of \cite[Theorem~1.1]{zhang2021} we can show that $m(r,e^{(m+1)p_2}/f^{(m+1)n-1})=O(r^{\rho})$. Together with this estimate and 
By Lemma~\ref{growthlemmapre} and the lemma on the logarithmic derivative, we obtain from \eqref{1Eq3.28a1} that $m(r,G_m)=O(r^{\rho})$ and hence $T(r,G_m)=O(r^{\rho})$, i.e., $G_m$ is a function of order $\leq \rho$. We denote $\varphi=G_m$.

%When $m=0$, $\varphi=G_0$ is a function of order at most~$\rho$, by substituting $f'=\kappa_0f+\varphi$ into \eqref{1Eq2.2} and recalling that $A_1=b_1b_2(B_2-B_1)$, we have
%\begin{equation}\label{1Eq2.2 fuji}
%f^{n}-\frac{nG_0}{B_2-B_1}f^{n-1}+\frac{B_2}{B_2-B_1}P-P'=b_1e^{p_1}.
%\end{equation}
%By Hayman's generalization of the theorem of Tumura--Clunie, there is a function $\gamma$ of order less than $k$ such that $(f-\gamma)^n=b_1e^{p_1}$. By substituting $f=b_1^{1/n}e^{p_1/n}+\gamma$ into \eqref{EQ1} and applying Borel's lemma to the resulting equation as in the proof of \cite[Theorem~1.1]{zhang2021} we easily show that there is some $1\leq j\leq n-1$ such that $t=(n-j)/n$. We omit those details.

%When $m\geq 1$, w

We denote $D_0=b_1^{1/n}$ and $D_j=\iota_{j-1}b_1^{-j}b_1^{1/n}$, $j=1,\cdots,m$. Since $\delta(p_1,\theta)=\cos k\theta$, we may choose $\theta_1=\pi/(2k)$. Since $\alpha>0$, we see that $S_{1,i}$ for $p_1$ and $S_{2,j}$ for $p_2$ defined in \eqref{expon10} are identical. For simplicity, we denote these sectors by $S_i$, $i=1,2,\cdots,2k$, and suppose that $\delta(p_1,\theta)\geq 0$ for $\theta$ in the sectors $S_{i}$ where $i=2,4,\cdots,2k$. Now we choose one $\theta$ such that $\delta(p_1,\theta)>0$ and let $z=re^{i\theta}\in S_{2k}$. We may suppose that $f(z)\not=0,\infty$ on the ray $z=re^{i\theta}$ and also that $0$ is not a pole of $D_0$ or $D_j$ for otherwise we do the translation $z\to z+c$ for a suitable constant $c$. Let $t_0=1/n$, $t_1=(\alpha-1)+1/n$, $\cdots$, $t_m=m(\alpha-1)+1/n$. By Theorem~\ref{maintheorem0} together with its remark, we may integrate $G_{m}=G_{m-1}'-\kappa_{m} G_{m-1}$ along the ray $z=re^{i\theta}$ to obtain
\begin{equation}\label{1Eq3.28a2 houf}
G_{m-1}=C_0D_me^{t_mp_1}+H_m,
\end{equation}
where $C_0$ is a constant and
\begin{equation}\label{1Eq3.28a2 houf0}
H_m=D_me^{t_mp_1}\int_0^{z}D_m^{-1}e^{-t_mp_1}\varphi ds-a_mD_me^{t_mp_1},
\end{equation}
where $a_m=a_m(\theta)$ is a constant such that $|H_m|=O(e^{r^{\rho}})$ along the ray $z=re^{i\theta}$.
Together with the definitions in \eqref{1Eq3.27a prepare0} and \eqref{1Eq3.27a prepare}, we may integrate the recursion formulas $G_j=G_{j-1}'-\kappa_jG_{j-1}$ from $j=m$ to $j=1$ along the ray $z=re^{i\theta}$ inductively and finally integrate $G_0=f'-\kappa_0f$ along the ray $z=re^{i\theta}$ to obtain
\begin{equation}\label{1Eq3.28a2}
f=b_1^{1/n}\sum_{i=0}^mc_i\left(\frac{b_2}{b_1}\right)^ie^{t_ip_1}+H_0,
\end{equation}
where $c_0$, $\cdots$, $c_m$ are constants and
\begin{equation}\label{1Eq3.28a5}
\begin{split}
H_{0}=b_1^{1/n}e^{t_0p_1}\int_{0}^{z}b_1^{-1/n}e^{-t_0p_1}H_1ds-a_{0}b_1^{1/n}e^{t_0p_1},
\end{split}
\end{equation}
where $a_0=a_{0}(\theta)$ is a constant such that $|H_0|=O(e^{r^{\rho}})$ along the ray $z=re^{i\theta}$.

Now, from the recursion formula $G_j=G_{j-1}'-\kappa_jG_{j-1}$, $j\geq 1$, and $G_0=f'-\kappa_0f$, we easily deduce that $f$ satisfies the linear differential equation
\begin{equation}\label{1Eq3.28lin}
f^{(m+1)}-\hat{t}_{m}f^{(m)}+\cdots+(-1)^{m+1}\hat{t}_0f=\varphi,
\end{equation}
where $\hat{t}_m$, $\hat{t}_{m-1}$, $\cdots$, $\hat{t}_0$ are functions formulated in terms of $\kappa_0$, $\cdots$, $\kappa_m$ and their derivatives. From the integrations in \eqref{1Eq3.28a2 houf0}--\eqref{1Eq3.28a5} we see that the general solutions of the corresponding homogeneous linear differential equation of \eqref{1Eq3.28lin} are defined on a finite-sheeted Riemann surface and are of the form $f=b_1^{1/n}\sum_{i=0}^mc_i(b_2/b_1)^ie^{t_ip_1}$, where $b_1^{1/n}$ is in general an algebraic function (see \cite{Katajamaki1993algebroid} for the theory of algebroid functions). Suppose that $\gamma$ is a particular solution of \eqref{1Eq3.28lin}. We may write the entire solution of \eqref{EQ1} as $f=b_1^{1/n}\sum_{i=0}^mc_i(b_2/b_1)^ie^{t_ip_1}+\gamma$. By an elementary series expansion analysis around the zeros of $b_1$, we conclude that $\gamma/b_1^{1/n}$ is a meromorphic function. This implies that $b_1$ is an $n$-square of some polynomial. Therefore, we can integrate $G_j=G_{j-1}'-\kappa_jG_{j-1}$ from $j=m$ to $j=1$ inductively and finally integrate $G_0=f'-\kappa_0f$ to obtain that $H_m$, $\cdots$, $H_0$ in \eqref{1Eq3.28a2 houf0}--\eqref{1Eq3.28a5} are meromorphic functions with at most finitely many poles. We choose $\gamma=H_0$. Recall that along the ray $z=re^{i\theta}$ such that $\delta(p_1,\theta)>0$ and $z=re^{i\theta}\in S_{2k}$, we have $|H_0|=O(e^{r^{\rho}})$. Denote $g=b_1^{1/n}\sum_{i=0}^mc_i(b_2/b_1)^ie^{t_ip_1}$. Then
\begin{equation*}
\begin{split}
g^n=b_1\sum_{k_0=0}^{mn}C_{k_0}\left(\frac{b_2}{b_1}\right)^{k_0}e^{(k_0t-k_0+1)p_1},
\end{split}
\end{equation*}
where $C_{k_0}=\sum_{\substack{j_0+\cdots+j_m=n,\\j_1+\cdots+mj_m=k_0}}\frac{n!}{j_0!j_1!\cdots j_m!}c_0^{j_0}c_1^{j_1}\cdots c_m^{j_m}$, $k_0=0,1,\cdots,mn$.
By Lemma~\ref{derivativelemma}, we may suppose that along the ray $z=re^{i\theta}$ we have $|f^{(j)}/f|=r^{j(k-1+\varepsilon)}$ for all $j>0$ for all sufficiently large $r$.
Then by writing $P$ in the form in \eqref{lemmaQ10} and using Lemma~\ref{derivativelemma} we may write each of the monomials in $P(z,f)$ of degree $n-j$, $1\leq j\leq n-1$, as a linear combination of exponential functions of the form $e^{[nk_{j}(\alpha -1)+n-j]p_1/n}$, $0\leq k_{j}\leq (n-j)m$, with coefficients $\beta_j$ of growth of type $O(e^{r^{\rho}})$ along the ray $z=re^{i\theta}$. Therefore, by substituting $f=g+H_0$ into \eqref{EQ1} we obtain,
when $m=0$,
\begin{equation}\label{1Eq3.29a}
\begin{split}
(c_0^n-1)b_1e^{p_1}-b_2e^{p_2}+\sum_{j=1}^{n}\beta_{j}e^{\frac{n-j}{n}p_1}=0,
\end{split}
\end{equation}
where $\beta_{j}$, $1\leq j\leq n$, are functions of growth of type $O(e^{r^{\rho}})$ along the ray $z=re^{i\theta}$, and, when $m\geq1$,
\begin{equation}\label{1Eq3.29}
\begin{split}
&(c_0^n-1)b_1e^{p_1}+[nc_0^{n-1}c_1e^{\alpha p_1-\alpha z^k}-e^{p_2-\alpha z^k}]b_2e^{\alpha z^k}\\
&+b_1\sum_{k_0=2}^{mn}C_{k_0}\left(\frac{b_2}{b_1}\right)^{k_0}e^{(k_0\alpha -k_0+1)p_1}+\sum_{j=1}^{n}\sum_{k_j=0}^{m(n-j)}\beta_{j,k_j}e^{[nk_{j}(\alpha-1)+n-j]p_1/n}=0,
\end{split}
\end{equation}
where $\beta_{j,k_j}$, $1\leq j\leq n$, $0\leq k_j\leq m(n-j)$, are functions of growth of type
$O(e^{r^{\rho}})$ along the ray $z=re^{i\theta}$. Note that in \eqref{1Eq3.29}, for each $j\geq0$, $[nk_{j}(\alpha-1)+n-j]/n$ decreases strictly as $k_j$ varies from $0$ to $(n-j)m$. Denote $\hat{c}(z)=nc_0^{n-1}c_1e^{\alpha p_1-\alpha z^k}-e^{p_2-\alpha z^k}$. Since $\rho_1=\rho(\hat{c}(z))\leq n-1$, then by our assumption on the $\hat{R}$-set, we have $e^{-r^{\rho_1+\varepsilon}}\leq |\hat{c}(z)|\leq e^{r^{\rho_1+\varepsilon}}$ for all $r$ outside a set of finite linear measure, provided that $\hat{c}(z)$ is not identically zero. We rewrite the left-hand sides of equation \eqref{1Eq3.29a} or \eqref{1Eq3.29} by combining the same exponential terms together. In so doing, by letting $z\to\infty$ along the ray $z=re^{i\theta}$ and comparing the growth on both sides of equation \eqref{1Eq3.29a} or \eqref{1Eq3.29} we conclude that the coefficient of the term which dominates the growth of the resulting equation along the ray $z=re^{i\theta}$ must be zero. Then from \eqref{1Eq3.29a} we have $c_0^n=1$; for \eqref{1Eq3.29}, noting that $k_0\alpha-k_0+1>1-k_0/mn\geq (n-1)/n$ for all $0\leq k_0 \leq m$ and also that $[nk_{j}(\alpha-1)+n-j]/n\leq (n-1)/n$ for all $j\geq 1$ and $k_j\geq 0$ since $(mn-1)/mn<\alpha \leq [(m+1)n-1]/(m+1)n$, we must have $c_0^{n}=1$, $nc_0^{n-1}c_1e^{\alpha p_1-\alpha z^k}-e^{p_2-\alpha z^k}\equiv0$ and further that $C_{k_0}\equiv0$ for all $2\leq k_0\leq m$ when $m\geq 2$. This implies that $p_2=\alpha p_1$ since $p_1(0)=p_2(0)=0$.

For a given $\epsilon>0$, denote by $S_{i,\epsilon}$, $i=1,\cdots,2k$, the corresponding sectors for $p_1$ defined in \eqref{expon10 fujia}. Now, since $\varphi$ and $D_j$ only have finitely may poles, by Theorem~\ref{maintheorem0} together with its remarks and looking at the calculations to obtain $H_0$ in \eqref{1Eq3.28a5}, we have, for $z\in S_{i,\epsilon}$, $i\in\{2,4,\cdots,2k\}$, such that $\delta(p_1,\theta)>0$, $H_0=b_1^{1/n}\sum_{l=0}^ma_{l,i}(b_2/b_1)^ie^{t_ip_1}+\gamma_{i}$, where $a_{l,i}$, $l=0,\cdots,m$, are some constants related to a sector $S_{i,\epsilon}$ and $|\gamma_i|=O(e^{r^{\rho}})$ uniformly for $z\in\overline{S}_{i,\epsilon}$ and $r$ is large. Of course, for $i=2k$, we have $a_{l,2k}=0$ for all $l$.
Therefore, by considering the growth of $f$ as to \eqref{1Eq3.29a} or \eqref{1Eq3.29} but now along the ray $z=re^{i\theta}$ such that $z\in S_{i,\epsilon}$, $i\in\{2,\cdots,2k\}$, and $\delta(p_1,\theta)>0$, we have $(c_0+a_{0,i})^n=1$ when $m=0$, $(c_0+a_{0,i})^n=n(c_0+a_{0,i})^{n-1}(c_1+a_{1,i})=1$ when $m=1$ and further that $\hat{C}_{k_0}=\sum_{\substack{j_0+\cdots+j_m=n,\\j_1+\cdots+mj_m=k_0}}\frac{n!}{j_0!j_1!\cdots j_m!}(c_0+a_{0,i})^{j_0}(c_1+a_{1,i})^{j_1}\cdots (c_m+a_{m,i})^{j_m}=0$ for $k_0=2,\cdots,m$ when $m\geq 2$.
For each $i\in\{2,\cdots,2k\}$, since $j_{k_0}\leq 1$ and $j_{k_0+1}=\cdots=j_m=0$, it is easy to see that there is an $\omega_{l,i}$ satisfying $\omega_{l,i}^n=1$ such that $c_l+a_{l,i}=\omega_{l,i} c_l$, $l=0,\cdots,m$. Note that when $i=2k$ we have $\omega_{l,i}=1$ for all $l=0,\cdots,m$.

Since $f=b_1^{1/n}\sum_{i=0}^mc_i(b_2/b_1)^ie^{t_ip_1}+H_0$, then by considering the growth of $P/f^{n-1}$ along the ray $z=re^{i\theta}$ such that $\delta(p_1,\theta)>0$, we find that $P/f^{n-1}=n\mu+\nu e^{-p_1/n}(1+o(1))$ for some functions $\mu$ and $\nu$ of order~$\leq \rho$. Further, by substituting $f=b_1^{1/n}\sum_{i=0}^mc_i(b_2/b_1)^ie^{t_ip_1}+H_0$ into $W_0=-(B_1P-P')/(nf^{n-1})$ and considering the growth along the ray $z=re^{i\theta}$ such that $\delta(p_1,\theta)>0$, we find that $W_0=\mu'-\kappa_0\mu+\nu_0e^{-p_1/n}(1+o(1))$ as $z\to\infty$ for some function~$\nu_0$. Since $\alpha<1$, we
have that $e^{jp_2}/f^{jn}= (c_0^nb_1)^{-j}e^{-j(1-\alpha')r^k}(1+o(1))$ for some $\alpha'>\alpha$, $j\geq 1$, as $z\to\infty$ along the ray $z=re^{i\theta}$ such that $\delta(p_1,\theta)>0$. Then we substitute $f=b_1^{1/n}\sum_{i=0}^mc_i(b_2/b_1)^ie^{t_ip_1}+H_0$ into the expressions for $G_0$, $G_1$, $\cdots$, $G_m$ in \eqref{1Eq3.28a1} inductively and consider the growth on both sides of the resulting equation along the ray $z=re^{i\theta}$ such that $\theta\in S_{i,\epsilon}$ and $\delta(p_1,\theta)>0$ as to
$\hat{c}(z)=nc_0^{n-1}c_1e^{\alpha p_1-\alpha z^k}-e^{p_2-\alpha z^k}$ in equation \eqref{1Eq3.29} to obtain: (1) if $\alpha<[(m+1)n-1]/[(m+1)n]$, then $G_m\equiv \mu^{(m+1)}-\hat{t}_{m}\mu^{(m)}+\cdots+(-1)^{m+1}\hat{t}_0\mu$; or (2) if $\alpha=[(m+1)n-1]/[(m+1)n]$, then $G_m=\iota_{m}/(\omega_{0,i} c_0b_1^{1/n})^{(m+1)n-1}+\mu^{(m+1)}-\hat{t}_{m}\mu^{(m)}+\cdots+(-1)^{m+1}\hat{t}_0\mu$. In the first case, $\mu$ is a particular solution of \eqref{1Eq3.28lin}; in the latter case, since $\iota_m\not\equiv0$, we see that $\omega_{0,i}=1$ for all $i=2,\cdots,2k$ and it follows that $\omega_{l,i}=1$ for all $l=0,\cdots,m$ and $i=2,\cdots,2k$. We conclude that $f=b_1^{1/n}\sum_{i=0}^mc_i(b_2/b_1)^ie^{t_ip_1}+\mu$ or $\omega_{0,i}=1$ for all $i=2,\cdots,2k$ and thus $f=b_1^{1/n}\sum_{i=0}^mc_i(b_2/b_1)^ie^{t_ip_1}+H_0$ with $a_{l,i}$, $l=0,1,\cdots,m$, $i=2,4,\cdots,2k$, being all equal to~$0$. Now the function $H_0$ satisfies $|H_0|=O(e^{r^{\rho}})$ uniformly for $z\in \overline{S}_{i,\epsilon}$ where $\delta(p_1,\theta)> 0$ and $r$ is large. Moreover, by Theorem~\ref{maintheorem0} together with its remarks we see that $|H_m|=O(e^{r^{\rho}})$ uniformly for $z\in \overline{S}_{i,\epsilon}$ where $\delta(p_1,\theta)< 0$ and $r$ is large, and from the integrations \eqref{1Eq3.28a2 houf0}--\eqref{1Eq3.28a5} we finally have that $|H_0|=O(e^{r^{\rho}})$ uniformly for $z\in \overline{S}_{i,\epsilon}$ where $\delta(p_1,\theta)< 0$ and $r$ is large. Therefore, for the given $\epsilon>0$, we have $|H_0(z)|=O(e^{r^{\rho}})$ for all $z\in \cup_{i=0}^{2n}\overline{S}_{i,\epsilon}$. Since $\epsilon>0$ can be arbitrarily small, then by the Phragm\'{e}n--Lindel\"{o}f theorem, we conclude that $H_0$ is a function of order~$\leq \rho$. We complete the proof.

%Then by Lemma~\ref{smallcoelemma}, we see that $H_0(z)$ is a small function of $f$. Since $c_0$, $\cdots$, $c_m$ are all nonzero, we may substitute $f=b_1^{1/n}\sum_{i=0}^mc_i(b_2/b_1)^ie^{t_ip_1}+\mu$ or $f=b_1^{1/n}\sum_{i=0}^mc_i(b_2/b_1)^ie^{t_ip_1}+H_0$ into \eqref{EQ1} again and apply Borel's lemma to the resulting equation as in the proof of \cite[Theorem~1.1]{zhang2021} to obtain that $\alpha$ is rational and $H_0$ has order~$\leq \rho$. We omit those details.

\end{proof}

\section{Further results about linear differential equations}\label{Application to complex dynamics}

In this section, we use the Phragm\'{e}n--Lindel\"{o}f theorem to prove a theorem on a class of meromorphic functions appearing in complex dynamics, namely \emph{class $N$}; see \cite{Baker1984,Bergweiler1992,Bergweiler1993}. For a meromorphic function $f$, we say that a point $z_0$ is a \emph{superattracting} fixed point of $f$ if $f(z_0)=z_0$ and $f'(z_0)=0$. A meromorphic function $f$ in $N$ has the following properties: $f$ has finitely many poles; $f'$ has finitely many multiple zeros; the superattracting fixed points of $f$ are zeros of $f'$ and vice versa, with finitely many exceptions; $f$ has finite order. Then a meromorphic function is in $N$ if and only if $f$ satisfies the first-order differential equation:
\begin{equation}\label{classNequatoin1}
f'(z)=\frac{q_1(z)}{q_2(z)}e^{q_3(z)}(f(z)-z),
\end{equation}
for some polynomials $q_1$, $q_2$ and $q_3$. By denoting $F(z)=f(z)-z$, then \eqref{classNequatoin1} becomes $F'=(q_1/q_2)e^{q_3}F-1$. We claim that $q_3$ must be a constant. More generally, consider the equation
\begin{equation}\label{classNequatoin2}
F'=R_1e^{q(z)}F+R_2,
\end{equation}
where $R_1$ and $R_2$ are two rational functions and $q(z)$ is a polynomial. We prove

\begin{theorem}\label{classNth}
Suppose that equation \eqref{classNequatoin2} has a nonconstant meromorphic solution of finite order. Then $q$ is a constant.
\end{theorem}

\renewcommand{\proofname}{Proof of Theorem~\ref{classNth}.}
\begin{proof}
The proof is a revised version of that in \cite{Baker1984} or in \cite{gundersenyang1998} since we need to deal with the poles of $f$. On the contrary, we suppose that $q(z)$ is a polynomial with degree $n\geq 1$. Then from \eqref{classNequatoin2} we see that $F$ is transcendental and has order of growth at least $n$. Now we recall the notation in section~\ref{introduction}. Denote the leading coefficient of $q(z)$ by $\alpha$ and write $\alpha=a+ib$ with $a,b$ real. Let $\theta_j$ be such that $\delta(q,\theta_j)=a\cos n\theta_j-b\sin n\theta_j=0$. Then $\theta_j$ divides the plane $\mathbb{C}$ into $2n$ sectors defined as in \eqref{expon10}. Without loss of generality, we may suppose that in the sectors $S_j$, where $j$ is odd, we have $\delta(q,\theta)>0$. Moreover, we suppose that all zeros and poles of $R_1$ and $R_2$ are contained in the disc $D:=\{z:|z|\leq r_0\}$ for some $r_0>0$. Then on the circle $|z|=r_0$, $g(z)$ is bounded. Let $\epsilon>0$ be given. We also recall the definition $S_{j,\epsilon}$ in \eqref{expon10 fujia}.

First, we consider $F(z)$ in a sector $\overline{S}_{j,\epsilon}$, where $j$ is odd. Let $\Gamma$ be a curve not intersecting the corresponding $R$--set of $F$ and contained in $\overline{S}_{j,\epsilon}\setminus D$. From the proof of Theorem~\ref{maintheorem0} we know that $|e^{q(z)}|\geq e^{c|z|^n}$ uniformly for all $z\in \overline{S}_{j,\epsilon}$ and some $c=c(j,\epsilon)>0$ and $|z|$ is large. By \cite[Proposition~5.12]{Laine1993} we know that $|F'/F|\leq r^{n_1}$ for some integer $n_1$ outside an $R$-set. We write equation \eqref{classNequatoin2} as
\begin{equation}\label{classNequatoin3}
e^{q}=\frac{1}{R_1}\frac{F'}{F}-\frac{R_2}{R_1}\frac{1}{F}.
\end{equation}
Then, for $z$ on the curve $\Gamma$ and $|z|=r$,
\begin{equation}\label{classNequatoin4}
\left|e^{q(z)}\right|\leq \left|\frac{1}{R_1(z)}\right|\left|\frac{F'(z)}{F(z)}\right|+|R_2(z)|\left|\frac{1}{F(z)}\right|\leq |z|^{n_1}+\left|\frac{R_2(z)}{R_1(z)}\right|\left|\frac{1}{F(z)}\right|.
\end{equation}
Therefore, for all $z$ on the curve $\Gamma$, we must have that $g(z)\to 0$ as $z\to\infty$ since otherwise the above inequality would imply that $e^{cr_k^n}\leq |r_k|^{n_1}+O(r_k^{m})$ as $r_k\to\infty$ for some infinite sequence $z_k=r_ke^{i\theta_k}$, where $m$ is an integer, which is impossible. Since $F(z)$ is analytic in the $R$--set corresponding to $F$ and outside $D$, we conclude by the maximum modulus principle that $F(z)\to 0$ as $z\to\infty$ for all $z\in \overline{S}_{j,\epsilon}\setminus D$, where $j$ is odd.

Second, we consider $F(z)$ in a sector $\overline{S}_{j,\epsilon}$, where $j$ is even. We write $R_2=P_1/P_2+P_3$ with three polynomials $P_1$, $P_2$ and $P_3$ such that $\deg(P_1)<\deg(P_2)$ and denote $\deg(P_3)=n_2$. It should be understood that $P_3\equiv0$ when $n_2=0$. Then along each ray $z=re^{i\theta}$ contained in the sector $\overline{S}_{j,\epsilon}$, we must have $|F'(z)|\leq |z|^{n_2+1}$ for all $z=re^{i\theta}$ and $|z|$ is large. In fact, if this is not true, then there is an infinite sequence $z_l=r_le^{i\theta_l}$ such that $|F'(z_l)|>|z_l|^{n_2+1}$ as $z_l\to\infty$ and moreover,
\begin{equation}\label{classNequatoin6be}
|F(z_l)|= \left|\int_{z_0}^{z_l}F'(t)dt-F(z_0)\right|\leq |F'(z_l)||z_l-z_0|+|F(z_0)|,
\end{equation}
which gives $|F(z_l)/F'(z_l)|\leq |z_l|+O(1)$. But then it follows by dividing $F'(z)$ on both sides of \eqref{classNequatoin2} that
\begin{equation*}
1=R_1e^{q}\frac{F}{F'}+R_2\frac{1}{F'},
\end{equation*}
and thus by the basic properties of $e^{q(z)}$ mentioned in the introduction we have
\begin{equation}\label{classNequatoin6}
1\leq \left|R_1(z_l)e^{q(z_l)}\right|\left|\frac{F(z_l)}{F'(z_l)}\right|+|R_2|\left|\frac{1}{F'(z_l)}\right|=o(1),
\end{equation}
a contradiction. Therefore, for all $z\in \overline{S}_{j,\epsilon}\setminus D$, $F'(z)$ is uniformly bounded by $|z|^{n_2+1}$ when $|z|$ is large. An elementary integration again shows that $|F(z)|\leq C|z|^{n_2+2}+O(1)$ holds uniformly for some $C$ and all $z\in \overline{S}_{j,\epsilon}\setminus D$.

From the above reasoning, we conclude that for a given $\epsilon>0$ and for all $z\in \overline{S}_{j,\epsilon}\setminus D$, $j=1,2,\cdots,2n$, and $|z|$ is large, we have $|F(z)|\leq |z|^{n_3+1}$ and $n_3=\max\{n_1,n_2+2\}$. Since $F(z)$ has only poles in the plane, there is some rational function $R(z)$ such that $R(z)\to0$ as $z\to\infty$ and that $h(z):=F(z)-R(z)$ is an entire function. Then the previous reasoning yields that $|h(z)|\leq |z|^{n_3}$ for all $z\in \overline{S}_{j,\epsilon}$, $j=1,\cdots,2n$, and $|z|$ is large. Since $\epsilon$ can be arbitrarily small, then the Phragm\'{e}n--Lindel\"{o}f theorem (see \cite[Theorem~7.3]{hollandasb}) implies that $h$ is a polynomial and so $F$ is a rational function, a contradiction. Therefore, $q$ must be a constant.

\end{proof}

%\section{Acknowledgements}

\noindent{\bf Acknowledgements.} The author would like to thank professor Janne Heittokangas of University of Eastern Finland (UEF) for introducing Steinmetz's paper~\cite{Steinmetz1978} to him during his visit to UEF in~2016--2018. The author would also like to thank professor Jianhua Zheng of Tsinghua University for pointing out a mistake in the proof of Theorem~\ref{maintheorem0} in the first version of the manuscript. The author also thank the referees for their very valuable suggestions for improving the original manuscript.

%$$$$
%\cite{goldbergo:08,halburdk:10JAMS} are removed
%$$$$

\bibliographystyle{amsplain}
%\bibliography{YY_datebase}

\end{document}